\documentclass[a4paper,11pt]{article}
\usepackage{graphicx} 
\usepackage{amsmath}
\usepackage{amssymb}
\usepackage{amsfonts}
\usepackage{amsthm}
\usepackage{parskip}
\usepackage{comment}
\usepackage{xcolor}
\usepackage{tikz}
\usepackage{yhmath}
\usepackage{enumitem}
\usepackage[top=22mm, bottom=23mm, left=22mm, right=22mm]{geometry}
\allowdisplaybreaks

\newtheorem{thm}{Theorem}[section]
\newtheorem{prop}[thm]{Proposition}
\newtheorem{lem}[thm]{Lemma}
\newtheorem{cor}[thm]{Corollary}

\newtheorem{ques}[thm]{Question}

\theoremstyle{definition}

\newtheorem{df}[thm]{\bf Definition}

\DeclareMathOperator{\Tr}{Tr}

\title{On the sum of the largest and smallest eigenvalues of graphs with high odd girth} 
\author{Fredy Yip\thanks{Trinity College, University of Cambridge, United Kingdom. Email: \textbf{fy276@cam.ac.uk}.}}
\date{}
\begin{document}

\maketitle 

\begin{abstract}
    The sum $\lambda_1 + \lambda_n$ of the maximum and minimum eigenvalues, and the odd girth of a graph both measure bipartiteness. We seek to relate these measures. In particular, for an odd integer $k\geq 3$, let $\gamma_k$ denote the supremum of $\frac{\lambda_1 + \lambda_n}{n}$ over graphs without odd cycles of length less than $k$. The example of the $k$-cycle $C_k$ shows that $\gamma_k\geq \Omega(k^{-3})$. In their recent work, Abiad, Taranchuk, and Van Veluw showed that $\gamma_k\leq O(k^{-1})$ and asked to determine the asymptotics of $\gamma_k$. Using approximation theory, we show that $\gamma_k\leq O(k^{-3}(\log k)^3)$, giving a tight upper bound up to a poly-logarithmic factor. 
\end{abstract}

\textbf{Keywords:} spectral graph theory, odd girth, measures of bipartiteness

\textbf{MSC:} 05C38, 15A18

\section{Introduction}

For a simple graph $G$ on $n$ vertices, its eigenvalues $\lambda_1\geq \dots\geq \lambda_n$ refer to the eigenvalues of its adjacency matrix $A(G)$. The odd girth $g_{\text{odd}}$ of $G$ refers to the length of the shortest odd cycle in $G$; we take $g_{\text{odd}} = \infty$ if there are no odd cycles in $G$. 

A graph $G$ is bipartite if and only if $g_{\text{odd}} = \infty$ (\emph{i.e.}~$G$ does not contain odd cycles). Likewise, a connected graph $G$ is bipartite if and only if $\lambda_1 + \lambda_n = 0$. Therefore, both the sum $\lambda_1 + \lambda_n$ of eigenvalues and the odd girth $g_{\text{odd}}$ form measures of the bipartiteness of $G$. 

It is therefore of interest to study if these measures can be directly related, especially in the ``bipartite limit" where $\lambda_1 + \lambda_n\rightarrow 0$ and $g_{\text{odd}}\rightarrow\infty$. Taking the $m$-fold blow-up of $G$ scales $\lambda_1 + \lambda_n$ by a factor of $m$, whilst preserving the odd girth $g_{\text{odd}}$. Therefore, we consider the normalisation $\frac{\lambda_1 + \lambda_n}{n}$ of the eigenvalue sum and its relationship with $g_{\text{odd}}$. 

For any fixed odd girth $g_{\text{odd}} = k$, it is possible that $\lambda_1 + \lambda_n = 0$, as shown, for instance, by the disjoint union of a $k$-cycle with an arbitrarily large balanced complete bipartite graph. Even requiring connectedness, it is possible for $\frac{\lambda_1 + \lambda_n}{n}\rightarrow 0$. 

We therefore focus instead on upper bounding $\frac{\lambda_1 + \lambda_n}{n}$ given $g_{\text{odd}} \geq k$. 

\begin{df}
    For an odd integer $k\geq 3$, let 
    \begin{equation*}
        \gamma_k = \sup\left\{\frac{\lambda_1 + \lambda_n}{n}\;\middle|\;g_{\text{odd}}\geq k\right\}
    \end{equation*}
    denote the supremum of $\frac{\lambda_1 + \lambda_n}{n}$ over all positive integers $n$ and graphs $G$ on $n$ vertices with $g_{\text{odd}}\geq k$. 
\end{df}

Abiad, Taranchuk, and Van Veluw~\cite{abiad2025} recently studied the asymptotics of $\gamma_k$ as $k\rightarrow\infty$. Prior to this, the $k = 5$ case, corresponding to triangle-free graphs, was studied by Csikv\'{a}ri~\cite{csikvari2022}, who gave the upper bound
\begin{equation} \label{triangle-free}
    \gamma_5\leq 3 - 2\sqrt{2}\approx 0.171572875. 
\end{equation}
A quantity related to $\lambda_1 + \lambda_n$ is the minimum eigenvalue $q_n$ of the signless Laplacian matrix $\overline{L}(G) = D(G) + A(G)$, where $D(G)$ is the diagonal matrix encoding the degree sequence of $G$. In a regular graph, we have $\lambda_1 + \lambda_n = q_n$. 

In the setting of regular triangle-free graphs, the analogous bound 
\begin{equation*}
    \frac{\lambda_1 + \lambda_n}{n} = \frac{q_n}{n}\leq 3 - 2\sqrt{2}
\end{equation*}
to (\ref{triangle-free}) was first given by Brandt~\cite{brandt1998}. Curiously, whilst the upper bound for $\frac{q_n}{n}$ has been improved to 0.1547 for general triangle-free graphs~\cite{balogh2023}, no corresponding improvement to (\ref{triangle-free}) has yet been given for the supremum $\gamma_5$ of $\frac{\lambda_1 + \lambda_n}{n}$. In our Concluding Remarks, we shall give a small improvement to (\ref{triangle-free}), namely
\begin{equation*}
    \gamma_5\leq \frac{1 - 14^{-1/3}}{1 + 14^{1/3}}\approx 0.171572529, 
\end{equation*}
which shall further be shown to be optimal without major changes in the proof approach. 

As noted by Abiad, Taranchuk, and Van Veluw~\cite{abiad2025}, a lower bound $\gamma_k = \Omega(k^{-3})$ could be given by the odd cycle $C_k$. Here, we have $n = k$, $\lambda_1 = 2$, $\lambda_n = 2\cos\left(\frac{(k - 1)/2}{k}\cdot 2\pi\right) = -2\cos\left(\frac{\pi}{k}\right)$, and therefore
\begin{equation*}
    \gamma_k\geq \frac{\lambda_1 + \lambda_n}{n} = \frac{2}{k}\left(1 - \cos\left(\frac{\pi}{k}\right)\right) = (\pi^2 - o(1))k^{-3} = \Omega(k^{-3}). 
\end{equation*}
On the other hand, Abiad, Taranchuk, and Van Veluw~\cite{abiad2025} established the upper bound $\gamma_k = O(k^{-1})$, and asked for the asymptotic behaviour of $\gamma_k$. 

\begin{thm}[Abiad, Taranchuk, and Van Veluw~\cite{abiad2025}, Theorem~1.3] \label{k^-1}
    $\gamma_k = O(k^{-1})$. 
\end{thm}

\begin{ques}[Abiad, Taranchuk, and Van Veluw~\cite{abiad2025}, Problem~1] \label{gamma_k ques}
    What is the asymptotic behaviour of $\gamma_k$, as $k\rightarrow\infty$?
\end{ques}

We shall make progress towards Question~\ref{gamma_k ques} by giving an upper bound which is tight up to a poly-logarithmic factor. 

\begin{thm} \label{main}
    $\gamma_k = O(k^{-3}(\log k)^3)$. 
\end{thm}

Our proof of Theorem~\ref{main} combines two arguments (Propositions~\ref{broad spectrum} and~\ref{high lambda_1}), which are effective over different regimes of $\frac{\lambda_1}{n}$. 

\begin{prop} \label{broad spectrum}
    Let $k\geq 100$ be an odd integer. In a graph $G$ on $n$ vertices with $g_{\text{odd}}\geq k$, if $\lambda_1\geq n/k^3$, then
    \begin{equation*}
        \frac{\lambda_1 + \lambda_n}{n}\leq \frac{4}{k^2}\cdot\frac{\lambda_1}{n}\left(\log\left(\frac{2n}{\lambda_1}\right)\right)^2. 
    \end{equation*}
\end{prop}

The condition that $\lambda_1\geq n/k^3$ is a minor technical assumption. In the case that $\lambda_1 < n/k^3$, we immediately have $\frac{\lambda_1 + \lambda_n}{n}\leq \frac{\lambda_1}{n} < k^{-3}$. 

Note that it follows from Proposition~\ref{broad spectrum} alone that $\gamma_k = O(k^{-2})$. 

\begin{prop} \label{high lambda_1}
    Let $k\geq 100$ be an odd integer. In a graph $G$ on $n$ vertices with $g_{\text{odd}}\geq k$, if $\lambda_1\geq 16n/k$, then
    \begin{equation*}
        \frac{\lambda_1 + \lambda_n}{n}\leq 4\cdot 2^{-\frac{k\lambda_1}{16n}}. 
    \end{equation*}
\end{prop}

Proposition~\ref{high lambda_1} shows an at least exponential decay of $\frac{\lambda_1 + \lambda_n}{n}$ as $\frac{k\lambda_1}{n}$ increases. For some absolute constants $c_1, c_2 > 0$, when $\lambda_1/n\notin [c_1k^{-1}(\log k)^{-2}, c_2 k^{-1}\log k]$, the stronger bound $\frac{\lambda_1 + \lambda_n}{n}\leq O(k^{-3})$ may, in fact, be shown from Propositions~\ref{broad spectrum} and~\ref{high lambda_1}. 

\section{Proof} \label{proof section}

\textbf{Notation.} We shall use $\log$ to refer to the natural logarithm. Let an \emph{odd polynomial} refer to a polynomial which only contains monomials of odd exponents. 

\begin{lem} \label{odd poly}
    Let $k\geq 3$ be an odd integer. In a graph $G$ on $n$ vertices with odd girth $g_{\text{odd}}\geq k$, for any odd polynomial $p$ of degree at most $k - 2$, we have
    \begin{equation} \label{trace}
        \sum_{i = 1}^n p(\lambda_i) = 0. 
    \end{equation}
\end{lem}

\begin{proof}
    As $G$ does not contain odd cycles of length less than $k$, $G$ cannot contain odd walks of length less than $k$. Therefore, for any positive odd integer $j\leq k - 2$, we have
    \begin{equation*}
        \sum_{i = 1}^n\lambda_i^j = \Tr(A(G)^j) = 0. 
    \end{equation*}
    The desired claim follows by taking an appropriate linear combination. 
\end{proof}

To prove Theorem~\ref{k^-1}, Abiad, Taranchuk, and Van Veluw~\cite{abiad2025} considered the trace identity
\begin{equation*}
    \sum_{i = 1}^n\lambda_i^{k - 2} = \Tr(A(G)^{k - 2}) = 0, 
\end{equation*}
corresponding to the monomial $p(x) = x^{k - 2}$ in Lemma~\ref{odd poly}. For any choice of an odd polynomial $p$ in Lemma~\ref{odd poly}, Equation~(\ref{trace}) is homogeneous. Hence, for any set of real numbers $\lambda_1\geq \dots\geq \lambda_n$ satisfying the statement of Lemma~\ref{odd poly}, and any $r > 0$, $r\lambda_1\geq \dots\geq r\lambda_n$ also satisfies the statement of Lemma~\ref{odd poly}. Therefore, Lemma~\ref{odd poly} is insufficient in itself to impose upper bounds on $\gamma_k$. To resolve this issue, Abiad, Taranchuk, and Van Veluw~\cite{abiad2025} considered, in addition, the inhomogeneous inequality
\begin{equation*}
    \sum_{i = 1}^n\lambda_i^2 = \Tr(A(G)^2) = 2e(G)\leq n\lambda_1
\end{equation*}
to obtain the upper bound $\gamma_k = O(k^{-1})$. 

Our proof of Proposition~\ref{broad spectrum} adapts the method of Janzer and Yip \cite{janzer2025}, applying Lemma~\ref{odd poly} with a suitable Chebyshev polynomial instead of a monomial. 

We recall well-known properties of Chebyshev polynomials of the first kind. 

\begin{lem}
    For any positive odd integer $j$, the Chebyshev polynomial $T_j$ of the first kind of degree $j$ satisfies the following properties: 
    \begin{itemize}
        \item $T_j$ is an odd polynomial, 
        \item $T_j(x)\geq -1$ for any $x\geq -1$, 
        \item $T_j(x) = \frac{1}{2}\left(\left(x + \sqrt{x^2 - 1}\right)^j + \left(x - \sqrt{x^2 - 1}\right)^j\right)$ for any $x\geq 1$. 
    \end{itemize}
\end{lem}

\begin{proof}[Proof of Proposition~\ref{broad spectrum}]
    As the desired bound is trivial if $G$ is the empty graph, we henceforth assume that this is not the case, so $\lambda_1 \geq |\lambda_n| > 0$. Applying Lemma~\ref{odd poly} with the odd polynomial $p(x) = x^2T_{k - 4}\left(\frac{x}{|\lambda_n|}\right)$, we have
    \begin{equation*}
        \lambda_1^2 T_{k-4}\left(\frac{\lambda_1}{|\lambda_n|}\right) = -\sum_{i = 2}^n \lambda_i^2 T_{k-4}\left(\frac{\lambda_i}{|\lambda_n|}\right)\leq \sum_{i = 2}^n\lambda_i^2 < \sum_{i = 1}^n\lambda_i^2 = \Tr(A(G)^2) = 2e(G)\leq n\lambda_1. 
    \end{equation*}
    Therefore, letting $r = \frac{\lambda_1}{|\lambda_n|} \geq 1$, we have
    \begin{equation*}
        \frac{n}{\lambda_1} > T_{k-4}(r) > \frac{1}{2}\left(r + \sqrt{r^2 - 1}\right)^{k - 4}\geq \frac{1}{2}\left(1 + \sqrt{2(r - 1)}\right)^{k - 4}.
    \end{equation*}
    As $\frac{n}{\lambda_1}\leq k^3\leq 2^{k - 5}$ for $k\geq 100$, we have $\sqrt{2(r - 1)}\leq 1$. As $e^{t/2}\leq 1 + t$ for $0\leq t\leq 1$, we have
    \begin{equation*}
        \frac{2n}{\lambda_1} > \exp\left(\frac{k - 4}{2}\sqrt{2(r - 1)}\right) > \exp\left(\frac{k}{2}\sqrt{r - 1}\right). 
    \end{equation*}
    Therefore, 
    \begin{equation*}
        r - 1 < \frac{4}{k^2}\left(\log\left(\frac{2n}{\lambda_1}\right)\right)^2. 
    \end{equation*}
    Hence, we have
    \begin{equation*}
        \frac{\lambda_1 + \lambda_n}{n} = \frac{|\lambda_n|}{n}(r - 1) \leq \frac{\lambda_1}{n}(r - 1) < \frac{4}{k^2}\cdot\frac{\lambda_1}{n}\left(\log\left(\frac{2n}{\lambda_1}\right)\right)^2. \qedhere
    \end{equation*}
\end{proof}

The salient properties of the Chebyshev polynomial $T_j$ in the proof above are its uniform boundedness on the interval $[-1, 1]$, and its rapid growth outside of this interval. A classical result due to Bernstein (\cite{bernstein1912}, Corollary 7) shows that the Chebyshev polynomial is in fact optimal for this purpose. 

To improve the upper bound on $\frac{\lambda_1 + \lambda_n}{n}$ in the case where $\frac{\lambda_1}{n}$ is large, we therefore turn to a choice of an odd polynomial $p$ which depends on the exact eigenvalues $\lambda_1, \dots, \lambda_n$ at hand. 

\begin{proof}[Proof of Proposition~\ref{high lambda_1}]
    As in the proof of Proposition~\ref{broad spectrum}, we may assume that $\lambda_1\geq |\lambda_n| > 0$. Let $\mu = \lambda_1/2$. Let $\lambda_1\geq\dots\geq \lambda_{d^+}\geq \mu > \lambda_{d^+ + 1}\geq \dots \geq \lambda_{n - d^-} > -\mu \geq \lambda_{n - d^- + 1}\geq \dots \geq \lambda_n$. Note that $d^+\geq 1, d^-\geq 0$. Let $d = d^+ + d^-$. As
    \begin{equation*}
        n\lambda_1\geq 2e(G) = \sum_{i = 1}^n \lambda_i^2\geq d\mu^2 = \frac{d}{4}\lambda_1^2, 
    \end{equation*}
    we have
    \begin{equation*}
        d\leq \frac{4n}{\lambda_1}\leq k/4. 
    \end{equation*}

    Let us apply Lemma~\ref{odd poly} with the odd polynomial $p(x) = x^{k - 4d^- - 2}\prod_{i = n - d^- + 1}^n(x^2 - \lambda_i^2)^2$. Note that $k - 4d^- - 2\geq k - 4d + 2\geq 2$. Note that $p(x)\geq 0$ for any $x\geq 0$. For $x\in [-\mu, \mu]$, we have
    \begin{equation*}
        |p(x)|\leq x^2\mu^{k - 4d^- - 4}\prod_{i = n - d^- + 1}^n|x^2 - \lambda_i^2|^2\leq x^2\mu^{k - 4d^- - 4}(\lambda_1^4)^{d^-} = x^2\lambda_1^{k - 4}2^{-k + 4d^- + 4}. 
    \end{equation*}
    Therefore, as $p(\lambda_i) = 0$ for $i\geq n - d^- + 1$, we have
    \begin{equation*}
        p(\lambda_1)\leq \sum_{i = 1}^{d^+}p(\lambda_i) = \left|\sum_{i = d^+ + 1}^{n - d^-}p(\lambda_i)\right|\leq \lambda_1^{k - 4}2^{-k + 4d^- + 4}\sum_{i = d^+ + 1}^{n - d^-}\lambda_i^2\leq n\lambda_1^{k - 3}2^{-k + 4d^- + 4}. 
    \end{equation*}
    On the other hand, we have
    \begin{equation*}
        p(\lambda_1) = \lambda_1^{k - 4d^- - 2}\prod_{i = n - d^- + 1}^n(\lambda_1^2 - \lambda_i^2)^2\geq \lambda_1^{k - 4d^- - 2}(\lambda_1^2 - \lambda_n^2)^{2d^-}. 
    \end{equation*}
    Combining the two preceding inequalities, we have
    \begin{equation*}
        \lambda_1^{k - 4d^- - 2}(\lambda_1^2 - \lambda_n^2)^{2d^-}\leq p(\lambda_1)\leq n\lambda_1^{k - 3}2^{-k + 4d^- + 4}. 
    \end{equation*}
    Dividing both sides by $\lambda_1^{k - 2}2^{4d^-}$, we have
    \begin{equation*}
        \left(\frac{1 - \left(\lambda_n/\lambda_1\right)^2}{4}\right)^{2d^-}\leq \frac{16n}{\lambda_1}2^{-k}. 
    \end{equation*}
    Therefore, noting that $k\geq 100$, we have
    \begin{equation*}
        \left(\frac{1 - |\lambda_n|/\lambda_1}{4}\right)^{2d^-}\leq \left(\frac{1 - \left(\lambda_n/\lambda_1\right)^2}{4}\right)^{2d^-}\leq \frac{16n}{\lambda_1}2^{-k}\leq k2^{-k}\leq 2^{-k/2}. 
    \end{equation*}
    Hence
    \begin{equation*}
        \frac{\lambda_1 + \lambda_n}{n}\leq \frac{4\lambda_1}{n}2^{-\frac{k}{4d^-}}\leq \frac{4\lambda_1}{n}2^{-\frac{k\lambda_1}{16n}}\leq 4\cdot 2^{-\frac{k\lambda_1}{16n}}. \qedhere
    \end{equation*}
\end{proof}

\begin{proof}[Proof of Theorem~\ref{main}]
    Since the conclusion is asymptotic as $k\rightarrow\infty$, we may assume that $k\geq 100$. Let $G$ be a graph on $n$ vertices with odd girth $g_{\text{odd}}\geq k$. We shall show that $\frac{\lambda_1 + \lambda_n}{n}\leq 6400k^{-3}(\log k)^3$. 
    
    If $\lambda_1\leq n/k^3$, then $\frac{\lambda_1 + \lambda_n}{n}\leq k^{-3}$, as needed. 

    If $n/k^3\leq \lambda_1\leq 100\frac{\log k}{k}n$, then by Proposition~\ref{broad spectrum}, as $k\geq 100$, we have
    \begin{equation*}
        \frac{\lambda_1 + \lambda_n}{n}\leq \frac{4}{k^2}\cdot\frac{\lambda_1}{n}\left(\log\left(\frac{2n}{\lambda_1}\right)\right)^2\leq \frac{400\log k}{k^3}(\log (2k^3))^2\leq \frac{400\log k}{k^3}(\log (k^4))^2 = \frac{6400(\log k)^3}{k^3}. 
    \end{equation*}

    And lastly, if $\lambda_1\geq 100\frac{\log k}{k}n$, then by Proposition~\ref{high lambda_1}, as $k\geq 100$, we have
    \begin{equation*}
        \frac{\lambda_1 + \lambda_n}{n}\leq 4\cdot 2^{-\frac{k\lambda_1}{16n}}\leq 4\cdot 2^{-\frac{100\log k}{16}}\leq 4k^{-4.332}\leq 4k^{-3}. \qedhere
    \end{equation*}
\end{proof}

\section{Concluding Remarks}

Throughout our analysis in Section~\ref{proof section}, the only properties we used about $\lambda_1\geq \dots\geq\lambda_n$ are the trace identities
\begin{equation} \label{conclusion odd trace}
    \sum_{i = 1}^n\lambda_i^j = 0,
\end{equation}
for odd positive integers $j\leq k - 2$, and 
\begin{equation} \label{conclusion 2 trace}
    \sum_{i = 1}^n\lambda_i^2 \leq n\lambda_1. 
\end{equation}
It is natural, therefore, to consider the supremum $\gamma_k'$ of $\frac{\lambda_1 + \lambda_n}{n}$ over positive integers $n$ and \emph{real numbers} $\lambda_1\geq \dots\geq\lambda_n$ satisfying (\ref{conclusion odd trace}) and (\ref{conclusion 2 trace}). It follows that $\gamma_k\leq \gamma_k'$. Our proof for Theorem~\ref{main} in fact gives $\gamma_k' = O(k^{-3}(\log k)^3)$. 

\begin{thm}
    $\Omega(k^{-3})\leq \gamma_k\leq \gamma_k' \leq O(k^{-3}(\log k)^3)$. 
\end{thm}

In addition to the asymptotic behaviour of $\gamma_k$, it may be of interest to study the asymptotic behaviour of $\gamma_k'$, which represents the optimal upper bound on $\gamma_k$ via (\ref{conclusion odd trace}) and (\ref{conclusion 2 trace}). 

\begin{ques}
    What is the asymptotic behaviour of $\gamma_k'$ as $k\rightarrow\infty$?
\end{ques}
\begin{ques}
    Does the asymptotic equality $\gamma_k = \Theta(\gamma_k')$ hold as $k\rightarrow\infty$? 
\end{ques}

As a first step towards understanding $\gamma_k'$, we finish by finding the exact value of $\gamma_5'$. 

\begin{thm} \label{gamma_5'}
    $\gamma_5' = \frac{1 - 14^{-1/3}}{1 + 14^{1/3}}$. 
\end{thm}

As $\gamma_5\leq \gamma_5'$, Theorem~\ref{gamma_5'} yields a very small improvement on Csikv\'{a}ri's~\cite{csikvari2022} upper bound $\gamma_5\leq 3 - 2\sqrt{2}$ for triangle-free graphs. 

\begin{cor}
    For any triangle-free graph $G$ on $n$ vertices, we have $\frac{\lambda_1 + \lambda_n}{n}\leq \frac{1 - 14^{-1/3}}{1 + 14^{1/3}}$. 
\end{cor}

More importantly, the tightness of Theorem~\ref{gamma_5'} shows that any further improvement to the upper bound on $\gamma_5$ must go beyond the trace identities (\ref{conclusion odd trace}) and (\ref{conclusion 2 trace}). Our proof of the upper bound in Theorem~\ref{gamma_5'} proceeds via a small modification of Csikv\'{a}ri's method in~\cite{csikvari2022}. As such, we defer the proof of Theorem~\ref{gamma_5'} to Appendix~\ref{gamma_5' appendix}. 

\textbf{Acknowledgements.} The author is grateful to Oliver Janzer for helpful comments and suggestions. The author would also like to thank the anonymous referees for their thoughtful and constructive comments. 

\bibliographystyle{abbrv}
\bibliography{mybib}

\appendix

\section{The value of $\gamma_5'$} \label{gamma_5' appendix}

We first prove the upper bound $\gamma_5'\leq \frac{1 - 14^{-1/3}}{1 + 14^{1/3}}$. 

\begin{lem} \label{start}
    Fix $\alpha > 1$, non-negative integers $\ell > m$ and $0\leq x < 1$. Let $s = m + x$. Consider the constrained optimisation problem (A) over $x_1, \dots, x_\ell\in [0, 1]$ where we seek to maximise 
    \begin{equation*}
        \sum_{i = 1}^\ell x_i^\alpha
    \end{equation*} given \begin{equation*}
        \sum_{i = 1}^\ell x_i = s. 
    \end{equation*}
    Up to permutation, the solutions to (A) are precisely $(x_1, \dots, x_\ell) = (1, 1, \dots, 1, x, 0, \dots, 0)$, where there are $m$ entries of 1. 
\end{lem}

\begin{proof}
    As we seek to maximise the value of a continuous function over a non-empty compact subset of $\mathbb{R}^\ell$, the constrained optimisation problem admits solutions. It suffices to show that any solution $(x_1, \dots, x_\ell)$ of (A) has at most one entry which is not zero or one. 
    
    For the sake of contradiction, let $(x_1, \dots, x_\ell)$ be a solution to (A) where $0 < x_1\leq x_2 < 1$. The function $t\mapsto (x_1 - t)^\alpha + (x_2 + t)^\alpha$ is smooth over $t\in [0, \min(x_1, 1 - x_2))$ with positive derivative for $t > 0$. Therefore, for some value of $t\in (0, \min(x_1, 1 - x_2))$, we have $(x_1 - t)^\alpha + (x_2 + t)^\alpha > x_1^\alpha + x_2^\alpha$. Considering the $\ell$-tuple $(x_1 - t, x_2 + t, x_3, x_4, \dots, x_\ell)$ contradicts the assumption that $(x_1, \dots, x_\ell)$ solves (A). 
\end{proof}

\begin{cor} \label{start cor}
    Fix $\alpha > 1$, non-negative integers $\ell, m$ and $0\leq x < 1$. Let $s = m + x$. If $x_1, \dots, x_\ell\in [0, 1]$ satisfy $\sum_{i = 1}^\ell x_i = s$, then
    \begin{equation*}
        \sum_{i = 1}^\ell x_i^\alpha\leq m + x^\alpha. 
    \end{equation*}
\end{cor}

\begin{proof}
    Since $\sum_{i = 1}^\ell x_i = s$ with $x_1, \dots, x_\ell\in [0, 1]$, we must have $\ell\geq s\geq m$. If $\ell = m$, we have $x_1 = \dots = x_\ell = 1$, and the desired inequality holds with equality. If $\ell > m$, the desired inequality follows from Lemma~\ref{start}. 
\end{proof}

Given $s\in \mathbb{R}$, let $f(s) = \lfloor s\rfloor + \{s\}^{3/2}$, where $\{s\}$ denotes the fractional part of $s$. 

\begin{lem}
    $\gamma_5' \leq \sup_{s\geq 1} \frac{1 - f(s)^{-1/3}}{1 + sf(s)^{-2/3}}$. 
\end{lem}

\begin{proof}
    Let $\lambda_1\geq \dots\geq\lambda_n$ be real numbers satisfying (\ref{conclusion odd trace}) (for $j = 1, 3$) and (\ref{conclusion 2 trace}). We would like to show that $\lambda_1 + \lambda_n\leq \frac{1 - f(s)^{-1/3}}{1 + sf(s)^{-2/3}} n$ for some $s\geq 1$. This inequality holds when $\lambda_1 = \dots = \lambda_n = 0$. Therefore, we assume that $\lambda_1 > 0 > \lambda_n$. Let $\lambda_1\geq \dots\geq \lambda_d \geq 0 > \lambda_{d + 1}\geq \dots\geq \lambda_n$. We have, 
    \begin{equation*}
        \sum_{i = d + 1}^n\lambda_i^2 \leq n\lambda_1 - \lambda_1^2. 
    \end{equation*}
    Take $s = \sum_{i = d + 1}^n\lambda_i^2/\lambda_n^2\geq 1$. Applying Corollary~\ref{start cor} to $(\lambda_{d + 1}^2/\lambda_n^2, \dots, \lambda_n^2/\lambda_n^2)$ and $\alpha = 3/2$, we have
    \begin{equation*}
        \sum_{i = d + 1}^n\lambda_i^3/\lambda_n^3\leq m + x^{3/2} = f(s), 
    \end{equation*}
    where $m = \lfloor s\rfloor$ and $x = \{s\}$. Therefore, 
    \begin{equation*}
        0 = \sum_{i = 1}^n\lambda_i^3 \geq \lambda_1^3 - \sum_{i = d + 1}^n|\lambda_i|^3\geq \lambda_1^3 - |\lambda_n|^3f(s). 
    \end{equation*}
    Therefore, we have
    \begin{align*}
        \frac{n\lambda_1 - \lambda_1^2}{|\lambda_n|^2}\geq s, \qquad\frac{\lambda_1^3}{|\lambda_n|^3}\leq f(s). 
    \end{align*}
    As a result, we have
    \begin{equation*}
        sf(s)^{-2/3}\leq \frac{n\lambda_1 - \lambda_1^2}{|\lambda_n|^2}\frac{|\lambda_n|^2}{\lambda_1^2} = \frac{n}{\lambda_1} - 1. 
    \end{equation*}
    Therefore, 
    \begin{equation*}
        \frac{\lambda_1 + \lambda_n}{n} = \frac{1 - |\lambda_n|/\lambda_1}{n/\lambda_1} \leq \frac{1 - f(s)^{-1/3}}{1 + sf(s)^{-2/3}}. \qedhere
    \end{equation*}
\end{proof}

Note that $f$ is an increasing function and $f(s)\leq s$, with equality if and only if $s$ in an integer. Replacing $f(s)$ by $s$ in the upper bound $\sup_{s\geq 1} \frac{1 - f(s)^{-1/3}}{1 + sf(s)^{-2/3}}$ in fact yields Csikv\'{a}ri's bound $3 - 2\sqrt{2}$. 

\begin{thm}
    $\gamma_5' \leq \frac{1 - 14^{-1/3}}{1 + 14^{1/3}}$. 
\end{thm}

\begin{proof}
    It suffices to show that $\frac{1 - f(s)^{-1/3}}{1 + sf(s)^{-2/3}}$ attains its maximum over $s\geq 1$ at $s = 14$. If not, let $\frac{1 - f(s)^{-1/3}}{1 + sf(s)^{-2/3}} > \frac{1 - 14^{-1/3}}{1 + 14^{1/3}}$ for some $s\geq 1$. As $f(s)\leq s$, we have 
    \begin{equation*}
        \frac{1 - s^{-1/3}}{1 + s^{1/3}} = \frac{1 - s^{-1/3}}{1 + s\cdot s^{-2/3}}\geq \frac{1 - f(s)^{-1/3}}{1 + sf(s)^{-2/3}} > \frac{1 - 14^{-1/3}}{1 + 14^{1/3}}. 
    \end{equation*}
    Multiplying both sides by $1 + s^{-1/3}$ and rearranging, we arrive at the quadratic inequality
    \begin{equation*}
        (1 - 14^{-1/3})s^{2/3} - (14^{1/3} + 14^{-1/3})s^{1/3} + (1 + 14^{1/3}) < 0
    \end{equation*}
    for $s^{-1/3}$. Solving this quadratic inequality gives 
    \begin{equation*}
        s^{1/3}\in \left(14^{1/3}, \frac{2}{14^{1/3} - 1} + 1\right). 
    \end{equation*}
    Therefore, 
    \begin{equation*}
        s\in \left(14, \left(\frac{2}{14^{1/3} - 1} + 1\right)^3\right)\subseteq (14, 14.143)\subseteq (14, 14.2). 
    \end{equation*}
    Let $x = s - 14\in (0, 0.2)$. We have $f(s) = 14 + x^{3/2} = 14 + x\cdot \sqrt x\leq 14 + x/2 = 14(1 + x/28)$. Therefore, 
    \begin{align*}
        f(s)^{-1/3}&\geq 14^{-1/3}(1 + x/28)^{-1/3}\geq 14^{-1/3}(1 - x/84), \\
        f(s)^{-2/3}&\geq 14^{-2/3}(1 + x/28)^{-2/3}\geq 14^{-2/3}(1 - x/42). 
    \end{align*}
    Hence, we have
    \begin{align*}
        \frac{1 - f(s)^{-1/3}}{1 + sf(s)^{-2/3}}&\leq \frac{1 - 14^{-1/3}(1 - x/84)}{1 + (14 + x)\cdot 14^{-2/3}(1 - x/42)}= \frac{1 - 14^{-1/3} + \frac{14^{-1/3}}{84}x}{1 + 14^{1/3} + \frac{14^{1/3}}{21}x - \frac{14^{1/3}}{588}x^2}\\
        &\leq \frac{1 - 14^{-1/3} + 0.005x}{1 + 14^{1/3} + 0.1x}\leq \frac{1 - 14^{-1/3}}{1 + 14^{1/3}}, 
    \end{align*}
    a contradiction. 
\end{proof}

To complete the proof of Theorem~\ref{gamma_5'}, it remains to show the lower bound $\gamma_5'\geq\frac{1 - 14^{-1/3}}{1 + 14^{1/3}}$. 

\begin{lem} \label{simple}
    For any positive integer $n$ and non-negative real numbers $c, d$, if $c^3n^{-2}\leq d\leq c^3$, then there exist non-negative real numbers $x_1, \dots, x_n$ such that 
    \begin{align*}
        \sum_{i = 1}^n x_i = c, \qquad\sum_{i = 1}^n x_i^3 = d. 
    \end{align*}
\end{lem}

\begin{proof}
    The subset $S = \{(x_1, \dots, x_n)\in \mathbb{R}_{\geq 0}^n|\sum_{i = 1}^n x_i = c\}$ of $\mathbb{R}^n$ is connected. At $(x_1, \dots, x_n) = (c, 0, \dots, 0)\in S$, we have $\sum_{i = 1}^n x_i^3 = c^3\geq d$, and at $(x_1, \dots, x_n) = (c/n, \dots, c/n)\in S$, we have $\sum_{i = 1}^n x_i^3 = c^3n^{-2}\leq d$. The desired result follows as $S$ is connected and $\sum_{i = 1}^n x_i^3$ is continuous in $(x_1, \dots, x_n)$. 
\end{proof}

We now show that $\gamma_5'\geq \frac{1 - 14^{-1/3}}{1 + 14^{1/3}}$ by constructing suitable real numbers $\lambda_1\geq \dots\geq\lambda_n$ satisfying (\ref{conclusion odd trace}) (for $j = 1, 3$) and (\ref{conclusion 2 trace}). 

\begin{thm}
    $\gamma_5' \geq \frac{1 - 14^{-1/3}}{1 + 14^{1/3}}$. 
\end{thm}

\begin{proof}
    Consider any $\epsilon\in (0, 1)$. 
    Let $x_1 = (14 - \epsilon)^{1/3}$, $x_{n - 13} = \dots = x_n = -1$. By Lemma~\ref{simple}, we may take $x_2, \dots, x_{n - 14}\geq 0$ such that
    \begin{align*}
        \sum_{i = 2}^{n - 14} x_i = 14 - (14 - \epsilon)^{1/3}, \qquad\sum_{i = 2}^{n - 14} x_i^3 = \epsilon, 
    \end{align*}
    whenever $n\geq n_\epsilon:= 15 + \sqrt{(14 - (14 - \epsilon)^{1/3})^3/\epsilon}$. We may assume, by permuting $x_2, \dots, x_{n - 14}$, that we have a decreasing sequence $x_1\geq \dots\geq x_n$. With this choice of $x_1, \dots, x_n$, we have
    \begin{align*}
        \sum_{i = 1}^n x_i = 0, \qquad\sum_{i = 1}^n x_i^3 = 0. 
    \end{align*}
    Now, by the Cauchy-Schwarz inequality, we have
    \begin{equation*}
       \sum_{i = 1}^n x_i^2 \leq 14^{2/3} + 14 + \sum_{i = 2}^{n - 14} x_i^2\leq 14^{2/3} + 14 + \sqrt{\left(\sum_{i = 2}^{n - 14} x_i\right)\left(\sum_{i = 2}^{n - 14} x_i^3\right)}\leq 14^{2/3} + 14 + \sqrt{14\epsilon}. 
    \end{equation*}
    Therefore, $\lambda_1\geq \dots\geq \lambda_n$ given by $\lambda_i = \frac{nx_1}{14^{2/3} + 14 + \sqrt{14\epsilon}}x_i$ satisfies (\ref{conclusion odd trace}) (for $j = 1, 3$) and (\ref{conclusion 2 trace}). Hence
    \begin{equation*}
        \gamma_5'\geq\frac{\lambda_1 + \lambda_n}{n} = \frac{x_1(x_1 + x_n)}{14^{2/3} + 14 + \sqrt{14\epsilon}} = \frac{(14 - \epsilon)^{2/3} - (14 - \epsilon)^{1/3}}{14^{2/3} + 14 + \sqrt{14\epsilon}}, 
    \end{equation*}
    and the right-hand side tends to $\frac{14^{2/3} - 14^{1/3}}{14^{2/3} + 14} = \frac{1 - 14^{-1/3}}{1 + 14^{1/3}}$ as $\epsilon\rightarrow 0$, which completes the proof. 
\end{proof}

\end{document}